\documentclass[12pt]{article}
\usepackage{mathrsfs}
\usepackage{bbm}
\usepackage{verbatim}
\usepackage{amsthm}
\usepackage{amsmath}
\usepackage{amssymb}
\usepackage{graphicx}
\usepackage{epstopdf}
\usepackage{titling}
\usepackage{enumerate}
\usepackage{algpseudocode}
\usepackage{algorithm}
\usepackage{algorithmicx}
\usepackage{tikz}
\usepackage{hyperref}

\newtheorem{theorem}{Theorem}[section]
\newtheorem{thm}[theorem]{Theorem}
\newtheorem{prop}[theorem]{Proposition}
\newtheorem{lem}[theorem]{Lemma}

\newtheorem{conj}[theorem]{Conjecture}
\newtheorem{defn}[theorem]{Definition}

\newtheorem{prob}[theorem]{Problem}
\newtheorem{rem}[theorem]{Remark}

\begin{document}

\title{Graphs with few $3$-cliques and $3$-anticliques are $3$-universal}

\author{Nati Linial\thanks{School of Computer Science and engineering, The Hebrew University of Jerusalem, Jerusalem 91904, Israel. Email: {\tt nati@cs.huji.ac.il}. Research supported in part by the Israel Science Foundation and by a USA-Israel BSF grant.}\and Avraham Morgenstern\thanks{Einstein Institute of mathematics, The Hebrew University of Jerusalem, Jerusalem 91904, Israel. Email: {\tt avraham.morgenstern@mail.huji.ac.il}}}


\maketitle

\setcounter{page}{1}

\vspace{-2em}

\begin{abstract}
For given integers $k, l$ we ask whether every large graph with a sufficiently {\em small} number of $k$-cliques and $k$-anticliques must contain an induced copy of every $l$-vertex graph. Here we prove this claim for $k=l=3$ with a sharp bound. A similar phenomenon is established as well for tournaments with $k=l=4$.
\end{abstract}

\section{Introduction}

We start by recalling the notion of universality.

\begin{defn}
A graph (resp. tournament) is called $l$-{\em universal} if it contains every $l$-vertex graph (tournament) as an induced subgraph (subtournament). 
\end{defn}

We next recall the celebrated Erd\H{o}s-Hajnal conjecture~\cite{EH} that we reformulate in a somewhat nonstandard form. As usual we denote by $\omega(G), \alpha(G)$ the clique, resp. anticlique number of the graph $G$.

\begin{conj}\label{conj:EH}[Erd\H{o}s-Hajnal]
For every integer $l$ there is an $\epsilon > 0$ such that every $n$-vertex graph $G$ with $\alpha(G), \omega(G) < n^{\epsilon}$ is $l$-universal.
\end{conj}

The largest size of a transitive subtournament of the tournament $T$ is denoted by $\text{tr}(T)$. The Erd\H{o}s-Hajnal conjecture for tournaments states:

\begin{conj}
For every integer $l$ there is an $\epsilon > 0$ such that every $n$-vertex tournament $T$ with $\text{tr}(T) < n^{\epsilon}$ is $l$-universal.
\end{conj}

As shown by Alon, Pach and Solymosi~\cite{Alon-Pach-Solymosi}, these two conjectures are equivalent.

The Erd\H{o}s-Hajnal conjecture in both its formulations posits that a graph (resp. a tournament) which satisfies a rather mild upper bound on largest clique and anticlique (resp. transitive set) must be $l$-universal. In this paper we ask the following 

\begin{prob}\label{prob:1} For given integers $k, l$ is every large graph with {\em few} $k$-cliques and $k$-anticliques necessarily $l$-universal? Similarly, is a large tournament with only few transitive subtournaments of order $k$ necessarily $l$-universal? 
\end{prob}

The answer for the graph problem with $k=l=3$ turns out to be positive, and we derive a sharp bound for this statement. For tournaments, the case of $k=3$ is trivial, but the range $k\ge 4$ turns out rather interesting. We prove that an upper bound on the number of transitive $4$-vertex subtournaments implies $4$-universality. As explained in the last section, this line of thought can be developed in numerous additional ways.

We need some definitions and notations which we state in the language of graphs. Their counterparts for tournaments are obvious. For a fixed $l$-vertex graph $H$ and an arbitrary graph $G$ we denote by $p(H,G)$ the probability that a randomly chosen set of $l$ vertices in $G$ induces a subgraph that is isomorphic to $H$. Given an integer $l$, we let ${\cal H}_l$ be the list of all $N=N_l$ isomorphism classes of $l$-vertex graphs. We refer to the vector $\pi_l(G)=(p(H,G))_{H \in {\cal H}_l}$ as the $l${\em-th local profile} of the graph $G$.

Below we use $\cal G$ to always denote a sequence of graphs $G_n$ with 
$|V(G_n)|\to\infty$. If the limit 
$\lim_n\lambda(G_n)$ exists, where
$\lambda$ is some graph parameter, we denote this limit by 
$\lambda({\cal G})$. Likewise 
$\bar{\lambda}({\cal G}):=\limsup_n\lambda(G_n)$ and 
$\underline{\lambda}({\cal G}):=\liminf_n\lambda(G_n)$.

For each $i=0,1,2,3$ there is exactly one graph $P_i\in{\cal H}_3$ that has $i$ edges, and we denote $p(P_i,{\cal G})$ by $p_i({\cal G})$, or simply by $p_i$ when ${\cal G}$ is clear from the context. We note that $p_0(G)$ (resp. $p_1(G)$) equals $p_3$ (resp. $p_2$) of its complement graph. For example, in our terminology, Goodman's well-known bound~\cite{goodman} takes the form: 

\begin{thm}\label{thm:goodman}[Goodman]
For every $\cal G$ there holds $\underline{p_0+p_3}({\cal G})\ge\frac{1}{4}$. 
\end{thm} 

Jacob Fox (personal communication) has observed that the answer to Problem~\ref{prob:1} is positive for some $l=\Omega(k)$. Namely, he found the following lemma whose proof appears in Section~\ref{sect:final}.
\begin{prop}\label{prop:2} If both $p(K_k,{\cal G})<2^{-{k\choose 2}}+\epsilon$ and $p(\overline{K}_k,{\cal G})<2^{-{k\choose 2}}+\epsilon$ then $\cal G$ is $ck$-universal, where $c>0$ is a universal constant. 
\end{prop}

We define ${\Pi}_l \subset \mathbb R^N$ as the set of all points $\pi \in \mathbb R^N$ for which there exists a sequence of graphs $\cal G$ with $\pi_l({\cal G})=\pi$. It is still a major open question to get a good description of these sets $\Pi_l$. In the present article we add some piece to what is known about $\Pi_3$. At this writing even $\Pi_3$ is not yet fully understood (but see~\cite{hlnps, razborov}). The state of our knowledge of $\Pi_l$ for $l \ge 4$ is really very limited, though some work already exists, e.g.,~\cite{fr, giraud, hlnps2, nikiforov, reiher, sperfeld, thomason}. Much of the recent progress in this area was achieved using Razborov's flag algebras method.\\
We say that $\cal G$ is $l$-{\em universal} if $\underline{p}(H,{\cal G})>0$ for every $H\in{\cal H}_l$.

Our main result is

\begin{thm}\label{thm:main}
There is a constant $\rho= 0.159181...$ such that every $\cal G$ with $\overline{p_0}({\cal G}),\overline{p_3}({\cal G})< \rho$ is $3$-universal. The bound is tight.\\
The number $\rho$ is defined as $\rho=6\theta^2(1-2\theta)$ where $\theta=0.427373...$ is the largest root of $\theta^3+\theta^2-\theta+\frac 16=0$.
\end{thm}

We prove this theorem in Section~\ref{sec:graphs}. In Section~\ref{sec:tournaments} we state and prove our results for tournaments. In Section~\ref{sect:final} we prove Proposition~\ref{prop:2} and mention several open questions.

\section{Proof of Theorem~\ref{thm:main} for graphs}\label{sec:graphs}

First, note that by Goodman's theorem \ref{thm:goodman},
\[\overline{p_0},\overline{p_3}\le\rho<\frac{1}{4} \implies \underline{p_0},\underline{p_3}>0.\]
It remains to prove that $\underline{p_1},\underline{p_2}>0$. By the above-mentioned symmetry between $p_1$ and $p_2$, it suffices to consider only $p_2$. 

By passing to a subsequence, if necessary, and arguing by contradiction, it suffices to consider only sequences $\cal G$ with $p_2({\cal G})=0$. By the graph removal lemma \cite{afks,cf}, an $n$-vertex graph $G$ with $p_2(G)=o(1)$ can be made $P_2$-free\footnote{Note that $P_2$ is a $3$-vertex path. This notation is not universally accepted, but hopefully no confusion is created.} by flipping only $o(n^2)$ edges\footnote{Here $o(1)$ means $o_n(1)$. In general, little-oh terms are taken w.r.t. to the order of the graph that tends to infinity.}. Since this changes $p_0(G)$ and $p_3(G)$ by only $o(1)$, we may apply this removal step to all $G\in{\cal G}$, and assume that every $G\in{\cal G}$ is $P_2$-free. But a graph is $P_2$-free iff it is a union of vertex disjoint cliques, so these are the only graphs we consider henceforth. Our goal is to prove that $\max(\overline{p_0},\overline{p_3})\ge\rho$ for such graphs.

We proceed with a series of reductions which allow us to make the following assumptions:

\begin{enumerate}
\item There is a bound on the number of cliques in all $G\in{\cal G}$.
\item Both limits $p_0({\cal G})$ and $p_3({\cal G})$ exist.
\item\label{bullet:3} $p_0({\cal G})=p_3({\cal G})$.
\end{enumerate}
Under these assumptions, the theorem follows from Lemma~\ref{lem:1} below.

It suffices to consider $n$-vertex graphs $G$ with only a bounded number of non-trivial cliques. For let us fix some $\epsilon>0$ and remove all the edges from every clique of size $<\epsilon n$. This leaves only $< \frac 1{\epsilon}$ non-trivial cliques in $G$ which now has the desired form. This changes the parameters $p_0(G),p_3(G)$ only by $O(\epsilon)$. By letting $\epsilon\to 0$ the reduction follows.

Our next reduction is to graphs $G$ with $|p_0(G)-p_3(G)| \le O(\frac 1n)$. Given the additional assumption that $p_0({\cal G})$ exists, this will imply $p_0({\cal G})=p_3({\cal G})$. Suppose that $p_0(G) - p_3(G) \gg \frac 1n$ for $G$ an $n$-vertex graph which is the union of vertex-disjoint cliques. We construct another $n$-vertex graph $G'$ with $p_0(G) > p_0(G'),~~ p_3(G) < p_3(G')$ and $|p_0(G') - p_3(G')| \le O(\frac 1n)$. This $G'$ is also the disjoint union of vertex-disjoint cliques and has no more cliques than $G$.

To construct $G'$ we sequentially move vertices from the smallest clique\footnote{If there is an isolated vertex in the graph, the corresponding clique gets eliminated, but this creates no problem in the argument.} in $G$ to the largest one, breaking ties arbitrarily, thereby changing $p_0$ and $p_3$ by at most $O(\frac 1n)$. We stop when $|p_0 - p_3| \le O(\frac 1n)$.

The case $p_0(G) < p_3(G)$ is similar, but even simpler. We sequentially isolate vertices until $|p_0 - p_3| \le O(\frac 1n)$. 

The last reduction is achieved by passing to a subsequence in which the limits $p_0=p_0({\cal G})$ and $p_3=p_3({\cal G})$ exist and are equal.
\qed
\vspace{3mm}

By passing to a subsequence if necessary we can fix the bound $r$ on the number of non-trivial cliques and the relative sizes $\alpha_1,\alpha_2,\ldots,\alpha_r\ge0$ of these cliques. In other words, we can now restrict ourselves to a sequence $\cal G$
whose $n$-th graph is $G_n=K_{\alpha_1n}\sqcup K_{\alpha_2n}\sqcup\ldots\sqcup K_{\alpha_rn}\sqcup \overline{K}_{\beta n}$ where $\alpha_1,\alpha_2,\ldots,\alpha_r\ge0$ and $\beta=1-\sum\alpha_i\ge 0$. We ignore issues of rounding $\alpha_j n$ to integral values since this affects the relevant parameters by only an additive $O(\frac 1n)$ term  \footnote{To see this designate one vertex in each nontrivial clique as ``special". The difference between the calculations below and the exact values comes only from triples that contain a special vertex.}. 
The next lemma deals with graphs of this particular structure.

\begin{lem}\label{lem:1}
Let $\alpha_1,\ldots,\alpha_r \ge 0$ and $\beta=1-\sum \alpha_i \ge 0$. Let
\[
p_3=\sum\alpha_i^3~~\text{and}~~p_0 = 6\sum_{i<j<k}\alpha_i\alpha_j\alpha_k+
6\beta\sum_{i<j}\alpha_i\alpha_j+3\beta^2\sum\alpha_i+\beta^3.
\]
If $p_0=p_3$ then
\[ p_0, p_3 \ge \rho=0.159181...\]
This bound is tight.
\end{lem}

\begin{proof}
We apply the Lagrange multipliers method to determine the smallest possible value of $\max(p_0,p_3)$ under the constraints $p_0=p_3$, $\alpha_i\ge 0$, $\sum\alpha_i\le 1$. (We eliminate the variable $\beta$ by substituting $\beta = 1 - \sum\alpha_i$). There are three cases to consider: 

\begin{itemize}
\item The minimum is attained in the interior of this region. We calculate the partial derivatives of the objective function $\frac{\partial p_3}{\partial \alpha_l}=3\alpha_l^2$ and the derivatives of the constraint \[\frac{\partial (p_3-p_0)}{\partial \alpha_l}=3\alpha_l^2-6\sum\limits_{i<j,~i,j\neq l}\alpha_i\alpha_j-6(1-\sum\limits_i\alpha_i)\sum\limits_{i\neq l}\alpha_i+6\sum\limits_{i<j}\alpha_i\alpha_j-3(1-\sum\limits_i\alpha_i)^2+\]\[
6(1-\sum\limits_i\alpha_i)\sum\limits_i\alpha_i+3(1-\sum\limits_i\alpha_i)^2=3\alpha_l^2+6\alpha_l\sum_{i\neq l}\alpha_i+6(1-\sum\alpha_i)\alpha_l=3\alpha_l(\alpha_l+2(1-\alpha_l)).\] The Lagrange multipliers method implies that at a critical point there holds $\frac{\partial p_3}{\partial \alpha_l}=\lambda\frac{\partial (p_3-p_0)}{\partial \alpha_l}$, where $\lambda$ is a Lagrange multiplier. Consequently $\alpha_l=\lambda(2-\alpha_l)$ for all $l$ (since we are working in the interior of our region, all $\alpha_l$ are positive). This is a linear equation, so all $\alpha_l$ are equal.
If $r\ge3$, then $p_3\le 3(\frac{1}{3})^3=\frac{1}{9}$, and $p_0\ge\frac{1}{4}-p_3>0.13$ (by Goodman's theorem), hence $p_0\neq p_3$, a contradiction.
If $r=1$, $p_3=\alpha_1^3=p_0=(1-\alpha_1)^3+3(1-\alpha_1)^2\alpha_1$. The solution is $\alpha_1=0.652704...$, for which $p_0=p_3=0.278...>\rho$.
If $r=2$, $p_3=2\alpha_1^3=p_0=(1-2\alpha_1)^3+3(1-2\alpha_1)^2\cdot2\alpha_1+6(1-2\alpha_1)\alpha_1^2$. The solution is $\alpha_1=0.442125...$, and $p_0=p_3=0.172848...>\rho$.
\item The minimum is attained when $\alpha_i=0$ for some $i$. This case is solved by removing this $\alpha_i$ using induction on $r$.
\item The minimum is attained when $\forall i~~ \alpha_i>0$, and $\sum\alpha_i=1$. We add the constraint $\sum\alpha_i=1$ to our Lagrange multipliers equations. This gives $\alpha_i^2=\lambda(2\alpha_i-\alpha_i^2)+\mu$. All $\alpha_i$ satisfy this quadratic equation, so they all take at most two different values, say $\alpha_1$ appears $s$ times and $\alpha_2$ appears $t$ times with $s\alpha_1+t\alpha_2=1$. We can assume that $s,t>0$, and $\alpha_1>\alpha_2>0$ since the case of equal $\alpha$'s was already treated above. 

If $s\ge 3$ then $p_3\le \frac{1}{9}$, and it follows (as before), that $p_0\neq p_3$, a contradiction.

If $s=1$, $p_3=\alpha_1^3+t\alpha_2^3$, $p_0=6{t\choose 2}\alpha_1\alpha_2^2+6{t\choose 3}\alpha_2^3$, and $\alpha_1=1-t\alpha_2$. Denote $x=t\alpha_2$. Here, $0<x<\frac{t}{t+1}$. We have $p_3=(1-x)^3+\frac{x^3}{t^2}$ and $p_0=3x^2-2x^3-\frac{3}{t}x^2+\frac{2}{t^2}x^3$. Let $\tau(x)=p_3-p_0$. The value of $x$ is determined by the equation $\tau(x)=0$. Note that $\tau(0)>0$, and $\tau'(x)<0$ for $0<x<\frac{t}{t+1}$. Therefore, $\tau$ is decreasing, and there is a unique solution for $\tau(x)=0$.

Now, $\tau(\frac{1}{3})=\frac{1}{27}+\frac{1}{3t}-\frac{1}{27t^2}>0$ implies that $x>\frac{1}{3}$. This implies that $p_0(x)>p_0(\frac{1}{3})$, since, $p_0'>0$. It remains to compute, for $t\ge 4$, $p_0(\frac{1}{3})>\frac{7}{27}-\frac{1}{3t}\ge\frac{7}{27}-\frac{1}{12}=0.175...>\rho$.

The case $s=t=1$ is vacuous, since $p_3>p_0=0$.

If $s=1, t=2$ the equation in $x=2\alpha_2$ is $1-3x+\frac{3}{2}x^2+\frac{3}{4}x^3=0$ with root at $x=0.469285...$, and $p_0=p_3=0.1753...>\rho$. 

If $s=1, t=3$ then $x=3\alpha_2$ satisfies $1-3x+x^2+\frac{8}{9}x^3=0$ so that $x=0.409632...$, and $p_0=p_3=0.2134...>\rho$. 

This concludes the case $s=1$, and the only remaining case to analyze is $s=2$. Again, $x:=t\alpha_2$. Here, $p_3=\frac{(1-x)^3}{4}+\frac{x^3}{t^2}$ and $p_0=\frac{3}{2}x-\frac{1}{2}x^3-\frac{3}{t}x^2+\frac{2}{t^2}x^3$. The range of $x$ is $0<x<\frac{t}{t+2}$. We first consider the case $t\ge 3$. Define $\tau(x)=p_3(x)-p_0(x)=\frac{1}{4}-\frac{9}{4}x+\frac{3}{4}x^2+\frac{1}{4}x^3+\frac{3}{t}x^2-\frac{1}{t^2}x^3$. Again, $\tau(0)>0$, and $\tau'<0$ for $0<x<\frac{t}{t+2}$, hence $\tau$ decreases. Also, $p_0$ increases, since $p_0'>0$.

Let $x_0=0.115749...$ be the solution in $[0,1]$ of the equation $\frac{1}{4}-\frac{9}{4}x+\frac{3}{4}x^2+\frac{1}{4}x^3=0$. $\tau(x_0)=\frac{3}{t}x_0^2-\frac{1}{t^2}x_0^3>0$. Since $\tau$ decreases, the solution for $\tau=0$ is bigger than $x_0$. Since $p_0$ increases, the optimal value of $p_0$ is larger than $p_0(x_0)=0.172848...-\frac{0.040193...}{t}+\frac{0.003102...}{t^2}\ge0.159450...>\rho$. 

If $s=2, t=1$, then $p_0=6\alpha_1^2(1-2\alpha_1)$, and $p_3=2\alpha_1^3+(1-2\alpha_1)^3$. Solving for $p_0=p_3$ gives $\alpha_1=0.234643...$ or $\alpha_1=0.427373...$. Since $\alpha_1>\alpha_2=1-2\alpha_1$, we have $\alpha_1=0.427373...$, and $p_0=p_3=\rho$. This example proves the tightness claim in the lemma.

Finally, $s=2, t=2$ gives $p_0=6(2\alpha_1^2 \alpha_2+2\alpha_1\alpha_2^2)=3\alpha_1(1-2\alpha_1)$ and $p_3=2\alpha_1^3+2\alpha_2^3=\frac{1}{4}-\frac{3}{2}\alpha_1+3\alpha_1^2$. Solving for $p_0=p_3$ with $\alpha_1>\alpha_2$ gives $\alpha_1=\frac{3+\sqrt{5}}{12}=0.436338...$, and $p_0=p_3=\frac 16>\rho$.
\qed\end{itemize}
\let\qed\relax\end{proof}

\section{On $4$-profiles of tournaments}\label{sec:tournaments}

As in the discussion above, we consider families $\cal T$ of tournaments of orders going to $\infty$ and discuss their $k$-local profiles. Likewise we define the limit values sets $\pi_l({\cal T})$ and the limit sets $\Pi_l$ of tournaments. The $3$-profiles of tournaments are easy and completely understood. There are just two $3$-vertex tournaments, one transitive and one cyclic with frequencies $t_3$ and $c_3=1-t_3$ respectively. It is well-known and easy to prove that for every $\cal T$ there holds $\underline{t_3}({\cal T}) \ge \frac 34$ and this is all there is to $3$-profiles of tournaments.

In this section we prove the analog of Theorem~\ref{thm:main} for tournaments and $k=4$. In addition we derive some information on $\Pi_4(\text{tournaments})$.

There are exactly four isomorphism types of $4$-vertex tournaments, see figure~\ref{fig:1}. Their names are as follows:

\begin{itemize}
\item
$T_4$ is the transitive $4$-tournament.
\item
$C_4$ is the (one and only) strongly connected $4$-tournament.
\item
In $W_4$ there is a cyclic triangle all three vertices of which arrow the fourth vertex.
\item
In $L_4$ one vertex arrows all the three vertices of a cyclic triangle
\end{itemize}

\begin{figure}
  \centering
  \includegraphics[width=160mm,height=40mm]{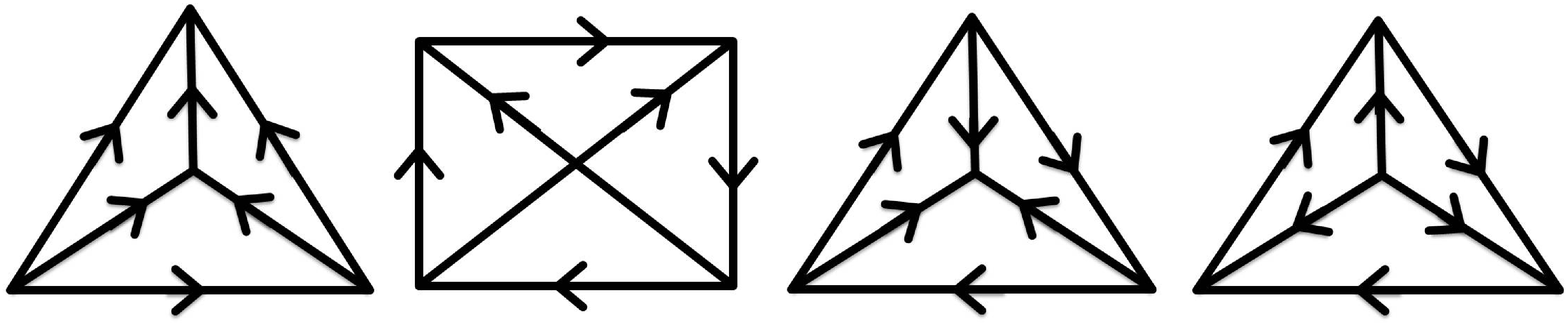}
  \caption{The four types of $4$-vertex tournaments (in order): $T_4,C_4,W_4,L_4$.}
  \label{fig:1}
\end{figure}

We use the shorthand
$t_4(T),c_4(T),w_4(T),l_4(T)$ for $p(T_4,T)$, etc., or even do not mention $T$ explicitly when clear from the context. Note that if the limits $t_4,c_4,w_4,l_4$ exist for some family of tournaments $\cal T$, then the limit fraction of cyclic triangles $c_3(\cal T)$ exists as well and equals $\frac{2c_4+w_4+l_4}{4}$. 

We recall the class ${\cal C}=C_n$ of {\em circular tournaments} of odd order $n$. The $n$ vertices of $C_n$ are equally spaced on the unit circle, with an edge $x\to y$ iff the clockwise arc from $x$ to $y$ is shorter than the counter clockwise arc between them.
We are now ready to state our theorem for tournaments:

\begin{theorem}
Every family of tournaments $\cal T$ for which $\overline{t_4}({\cal T})<\frac{1}{2}$ is $4$-universal. Moreover, $\underline{w_4}({\cal T}),\underline{l_4}({\cal T})\ge \frac 12 - \overline{t_4}({\cal T})$. Also $\underline{c_4}\ge\frac16$ when $\overline{t_4}({\cal T})\le\frac{1}{2}$.\\
The circular tournaments satisfy $t_4({\cal C})=\frac{1}{2}$ and yet $l_4=w_4=0$.
\end{theorem}

\begin{rem}
We do not know whether the inequality $\underline{c_4}\ge\frac16$ is tight, and so we ask how small $c_4({\cal T})$ can be when $t_4({\cal T})=\frac12$. A similar question is presented in remark~\ref{rem:1}.
\end{rem}

\begin{proof}
The theorem follows from the proposition below. In more detail, the positivity of $t_4, l_4$, and $w_4$ follows from items (\ref{eq:b}), (\ref{eq:c}), and (\ref{eq:d}) respectively. The lower bound on $c_4$ follows by combining (\ref{eq:e}) with the equality $c_3=\frac{1-t_4+c_4}{4}$.
\end{proof}

\begin{prop}\label{prop:ineq}
The following inequalities hold for every $(t_4,c_4,w_4,l_4)\in\Pi_4$. 
\begin{equation}\label{eq:a} c_4\le t_4\end{equation}
\begin{equation}\label{eq:b} \frac{3}{8}\le t_4 \end{equation}
\begin{equation}\label{eq:c} t_4+l_4\ge\frac{1}{2} \end{equation}
\begin{equation}\label{eq:d} t_4+w_4\ge\frac{1}{2} \end{equation}
All the above inequalities are tight. In addition:
\begin{equation}\label{eq:e} c_4\ge 6c_3^2.\end{equation}
\end{prop}

\begin{rem}
These inequalities, four linear and one quadratic, provide some information on the set $\Pi_4$. It would be interesting to derive a full description of $\Pi_4$.
\end{rem}

\begin{rem}\label{rem:1}
We still do not know how tight inequality~(\ref{eq:e}) is and we ask how small $c_4({\cal T})$ can be, given $c_3({\cal T})$. It is not difficult to see that this question is equivalent to the problem of minimizing $t_4({\cal T})$ given $t_3({\cal T})$, which is analogues to an interesting question about graphs: Let $2\le s<r$, given $p(K_s,{\cal G})$ how small can $p(K_r,{\cal G})$ be? (This question is stated in its general form in~\cite{hlnps} though it was probably posed earlier.) Razborov's recent solution for $s=2,r=3$~\cite{razborov} was a major achievement in local graph theory. The problem was subsequently solved for $s=2,r=4$ by Nikiforov~\cite{nikiforov}, and for $s=2$, and general $r$ by Reiher~\cite{reiher}. To the best of our knowledge, the problem remains open for $s\ge3$.
\end{rem}

\begin{proof}[Proof of Proposition~\ref{prop:ineq}]
Inequality (\ref{eq:a}): Recall that $t_3\ge\frac{3}{4}$, and $c_3\le \frac{1}{4}$. The inequality follows, since, $c_3=\frac{2c_4+l_4+w_4}{4}$. This holds with equality for the circular tournaments ${\cal C}$, for which $t_4=c_4=\frac{1}{2}$, and $l_4=w_4=0$.
\\
Inequality (\ref{eq:b}) follows by applying the inequality $t_3\ge\frac{3}{4}$ to the out-set of every vertex. To see that there are always at least $\frac{3}{4}\sum{d^+(x)\choose 3}$ transitive $4$-vertex subtournaments, count for each vertex $x$, the number of transitive triangles among the $d^+(x)$ out-neighbors of $x$. The inequality follows now from the convexity of the function ${t\choose 3}$. Equality holds, e.g., for random tournaments.
\\
Inequalities (\ref{eq:c}) and (\ref{eq:d}) are equivalent, of course. We prove the latter. Clearly ${n\choose 4}(t_4+w_4) = \sum_{x\in V} {d^+(x)\choose 3}$. Again a simple convexity argument yields the inequality and equality holds for random tournaments.
\\
Inequality (\ref{eq:e}): Note that two cyclic triangles sharing a common edge necessarily form a $C_4$. Let us denote by $S$ the set of cyclic triangles in an $n$-vertex tournament $T$, and for an edge $e$, $S_e=\{s\in S: e\subset s\}$. Then $c_4{n \choose 4}=\sum_e{|S_e|\choose 2}\ge {n\choose 2}{\sum|S_e|/{n\choose 2}\choose 2}={n\choose 2}{3c_3{n\choose 3}/{n\choose 2}\choose 2}=(6+o(1))c_3^2{n\choose 4}$. The inequality follows.
\end{proof}

\section{Further directions and discussion}\label{sect:final}

The following questions suggest themselves:
\begin{enumerate}
\item\label{he:ty}
Is there some $\epsilon>0$ such that every graph family with $p_0,p_3<\frac{1}{8}+\epsilon$ is $4$-universal? As observed by Mykhaylo Tyomkyn (personal communication), no such condition yields $5$-universality, see below.
\item
Is there some $\epsilon>0$ such that every graph family $\cal G$ with $p(K_4,{\cal G}),p(\overline{K}_4,{\cal G})<\frac{1}{64}+\epsilon$ is $l$-universal for some values of $l\ge 3$?
\item\label{klr}
What are the triples $k,l,r$ for which there exists an $\epsilon>0$ such that the conditions $p(K_k,{\cal G})<2^{-{k\choose 2}}+\epsilon $ and $p(\overline{K}_r,{\cal G})<2^{-{r\choose 2}}+\epsilon$ imply $l$-universality?
\item
Is there some $\epsilon>0$ such that every tournaments family with $t_4<\frac{3}{8}+\epsilon$ is necessarily $5$-universal? What about $l$-universality for bigger $l$?
\item
Does $t_5<\frac{5!}{2^{10}}+\epsilon$ imply $l$-universality for some values of $l \ge 4$?
\item\label{tour_kl}
For which integers $k,l$ does there exist an $\epsilon>0$ such that every tournament satisfying $t_k<\frac{k!}{2^{k\choose 2}}+\epsilon $ is $l$-universal? (Here $t_k$ is the proportion of transitive $k$-vertex subtournaments).
\item
Jacob Fox has raised the question whether problem~\ref{prob:1} can have a positive answer only with $l=O(k)$. As he pointed out, this would follow from the existence of a large $k$-clique-free graph $G$ with $p(\overline{K}_k,G) < 2^{-{k\choose 2}}+o_{|G|}(1)$. We note that it is an old and intriguing problem how small $p(\overline{K}_k,G)$ can be for a large $k$-clique-free graph. See~\cite{erdos} and the recent work~\cite{pikhurko:vaughan}.
\item\label{t_k-t_l}
What are the possible values of $t_k(\cal T)$, given the value of $t_l(\cal T)$? Here $\cal T$ is a family of tournaments and $k > l\ge 3$ are integers. The first interesting (and open) case is $k=4,~l=3$.
\item
In this article we discuss how the paucity of small homogeneous sets implies universality in graphs and in tournaments. It is conceivable that these two problem sets can be connected, perhaps in the spirit of Alon, Pach and Solymosi~\cite{Alon-Pach-Solymosi}, but we do not know how or whether this can be done. Specifically, can some connection can be established between items \ref{klr} (say, with $k=r$) and \ref{tour_kl} above? 
\end{enumerate}

We now present Jacob Fox's proof of Proposition~\ref{prop:2} which uses the following result of Pr\"{o}mel and R\"{o}dl. For a simpler proof of the results from~\cite{promel:rodl} with improved constants see~\cite{fox:sudakov}.

\begin{prop}\label{thm:pr}
For every $a>0$ there is a $b>0$ such that every $n$-vertex graph $G$ with $\alpha(G), \omega(G) < a\log n$ is $b\log n$-universal.
\end{prop}

\begin{proof}[Proof of Proposition~\ref{prop:2}]
Let $G\in{\cal G}$, be an $n$-vertex graph. Select $H$ as a random subgraph of $G$ with $m=2^{k/4}$ vertices. The expected number of $k$-cliques and $k$-anticliques in $H$ is at most ${m\choose k}(p(K_k,G)+p(\overline{K}_k,G))\le m^k\cdot2\cdot(2^{-{k\choose 2}}+\epsilon+o_n(1))$ which can be made smaller than $1$ by making $k$ large enough. Therefore, such an $H$ exists with $\alpha(H), \omega(H)< k$. Proposition~\ref{thm:pr} implies that $H$ is $ck$-universal, where $c$ is a constant. This clearly implies that $G$ is $ck$-universal, as claimed. 
\end{proof}

Two comments are in order here: (i) The above argument applies, with some minor modifications to item~\ref{klr} above and yields that if $r$ and $k$ are of the same order, then $l$ could be of the same order as well. (ii) Jacob Fox has pointed out that the methods of \cite{fox:sudakov} can be adapted to yield an analogue of the simple proof for Proposition~\ref{thm:pr} for tournaments, namely

\begin{prop}\label{thm:pr2}
For every $a>0$ there is a $b>0$ such that every $n$-vertex tournament $T$ with $\text{tr}(T) < a\log n$ is $b\log n$-universal.
\end{prop}

From this we can easily deduce:

\begin{lem}\label{lem:2}
Every family of tournaments $\cal T$ for which $t_k({\cal T})<k!2^{-{k\choose2}}+\epsilon$ is $ck$-universal for some absolute $c > 0$. 
\end{lem}

\begin{proof}
Apply the above proof of Proposition~\ref{prop:2} to an arbitrary large $T\in{\cal T}$. It shows that a random subset $S$ of $2^{k/4}$ vertices in $T$ contains no transitive $k$-vertex subtournament. By Proposition~\ref{thm:pr2} the tournament induced on $S$, and therefore the whole of $T$ is $ck$ universal.
\end{proof}

Mykhaylo Tyomkyn (personal communication) found the following recursive construction which is not $5$-universal even though $p_0,p_3\le\frac{1}{8}$.
Let $G_1=C_5$ be the pentagon graph. To construct $G_n$, take $5$ {\em blocks} each being a copy of $G_{n-1}$ and connect every two consecutive (modulo $5$) copies by a complete bipartite graph. It is easy to see that the graph $G_n$ is self-complementary, and $p_3(G_n)=\frac{1}{25}p_3(G_{n-1})+\frac{6}{25}e(G_{n-1})$ where $e(G_{n-1})=\frac{1}{2}$ is the edge density of $G_{n-1}$. It follows that $\lim_n p_3(G_n)=\frac{1}{8}=\lim_n p_0(G_n)$. On the other hand, this family of graphs is not $5$-universal. We prove by induction that $G_n$ has no induced copy of a $5$-vertex path $x_1,\ldots,x_5$. By induction not all $5$ vertices are in the same block. Also, if they all reside in different blocks then $x_1$ and $x_5$ are adjacent which is impossible. So let $x_u, x_v$ with $u<v$ reside in the same block and $x_w$ be in a neighboring block. Then necessarily $u=w-1, v=w+1$. But at least one of $x_u, x_v$ has a nighbor other than $x_w$ and this vertex cannot be fit into any of the blocks. 

{\bf Notes added in proof:} \begin{enumerate}
\item In a follow up paper, Hefetz and Tyomkin~\cite{hefetz:tyomkin} settle Problem~\ref{he:ty} in the above list and make several additional interesting contributions to this area.
\item We have recently made some progress on Problem~\ref{t_k-t_l} above. We intend to publish our results soon. 
\end{enumerate}

\section{acknowledgement}
We are grateful to Jacob Fox and Mykhaylo Tyomkyn for generously sharing their insights with us.

\bibliographystyle{amsplain}

\end{document}